\documentclass{article}
\usepackage{amssymb,amsmath,amsthm,hyperref,comment}

\newtheorem{theorem}{Theorem}[section]
\newtheorem{lemma}[theorem]{Lemma}
\newtheorem{proposition}[theorem]{Proposition}

\theoremstyle{definition}

\hypersetup{
    colorlinks=true,
    linkcolor=blue,
    urlcolor=blue
}

\title{The cardinality of a set containing the pairwise sums of four positive integers}
\author{Wouter van Doorn and John Erlbacher}
\date{}

\begin{document}
\maketitle

\begin{abstract}
Choi, Erd\H{o}s and Szemer\'edi showed that there exists an absolute constant $C$ such that for all subsets $A \subseteq \{1, 2, \ldots, 2n\}$ with at least $n+C$ elements, there exist four distinct positive integers whose pairwise sums are all contained in $A$. A proof that one can take $C = 3166$ was recently sketched by the first author, and here we show that we actually have $C = 4$ for all $n \ge 6$, which is optimal. The proof we present was originally conceived of by AI, with the final result proved and formalized by \textnormal{\href{https://aristotle.harmonic.fun/}{Aristotle}}, the formal reasoning agent developed by \textnormal{\href{https://www.harmonic.fun/}{Harmonic}}~\cite{ari}.
\end{abstract}

\section{Introduction}
Let $n$ and $k \ge 3$ be positive integers. We then define $h_k(n)$ to be the smallest integer such that for all sets $A \subseteq \{1, 2, \ldots, 2n\}$ with $\lvert A \rvert \ge n + h_k(n)$, there exist distinct positive integers $b_1, b_2, \ldots, b_k$ with $b_i + b_j \in A$ for $1 \le i < j \le k$. We hereby note that the $b_i$ themselves do not have to be elements of $A$. Various results on $h_k(n)$ were announced by Erd\H{o}s in 1972~\cite[p. 83]{Er72} and published three years later in a paper by Choi, Erd\H{o}s and Szemer\'edi~\cite{ces}. Estimating $h_k(n)$ (or the closely related function $g_k(n)$ where the $b_i$ are not required to be positive; see~\cite[Section 3]{wvd} for more information) is now listed as Erd\H{o}s Problem \#866 on Bloom's website~\cite{bloom}. \\

As for the results by Choi, Erd\H{o}s and Szemer\'edi, they managed to find the growth rates of $h_k(n)$ for all $k \in \{3, 4, 5, 6\}$, as well as upper and lower bounds on $h_k(n)$ for $k \ge 7$. On the other hand, the growth rate of $h_7(n)$ is still not known and neither are asymptotics for $h_k(n)$ for any $k \ge 5$. Specializing to the case $k = 4$, it was proved in~\cite{ces} that $h_4(n)$ is uniformly bounded by some absolute constant. The first author then sketched a proof in~\cite{wvd} that one can take $3166$ for this absolute constant, which was further lowered to $2270$ by the automated theorem proving tool Aristotle from Harmonic~\cite{ari}. The aim of the current paper is to show that we actually have the equality $h_4(n) = 4$ for all $n \ge 6$. In fact, we will prove that this equality holds for all large enough $n$, and verify the remaining ones computationally.

\section{Declaration of AI usage} \label{ai}
The original Lean formalization of the results in~\cite{wvd} (which is still available at~\cite{git0}) was obtained by Aristotle (Harmonic). Claude (Anthropic) then built on these results and formally proved that $h_4(n) = 4$ holds for all $n \ge 331,\!777$. We then gave this formalization to Aristotle again, to clean it up and to see if it could improve it further without any additional human input. This was hugely successful, and it autonomously managed to bring the bound $n \ge 331,\!777$ all the way down to $n \ge 62$. In doing so, Aristotle simplified the proof significantly, improving the result along with it. In the end, the (upper bound) proof consists of five distinct cases that we outline in Section~\ref{strat}. \\

In any case, after these simplifications and improvements, the remaining values were now small enough that they could be brute-forced, thereby showing that $h_4(n) = 4$ holds for all $n \ge 6$. Afterwards, the Aristotle-Claude proof was read and digested by the authors, and the end result is completely human-written. The final formalization of this paper is available at~\cite{git1}.

\section{Main theorem and proof strategy} \label{strat}
In this paper we aim to prove the following result\footnote{Note that the bound of $n \ge 3000$ is quite a bit larger than the $n \ge 62$ we mentioned in Section~\ref{ai}. This is due to various human simplifications and it is still small enough that it is possible to check the smaller $n$ with a computer and some patience.}. 

\begin{proposition} \label{mainprop}
For all $n \ge 3000$ we have $h_4(n) \le 4$. That is, for all $n \ge 3000$ and all sets $A \subseteq \{1, 2, \ldots, 2n\}$ with $\lvert A \rvert \ge n+4$, there exist four distinct positive integers $b_1, b_2, b_3, b_4$ with $b_i + b_j \in A$ for all $1 \le i < j \le 4$.
\end{proposition}

In the other direction, the lower bound $h_4(n) \ge 4$ is not hard to prove for $n \ge 3$, and we will do so in Section~\ref{lower}. As it is furthermore possible to computationally verify the value of $h_4(n)$ for the remaining $n < 3000$, we get the following main result.

\begin{theorem} \label{main}
For all $n \ge 6$ we have $h_4(n) = 4$. 
\end{theorem}

For a quick overview of the proof of Theorem~\ref{main}, let $$t := \left \lvert \{1, 3, \ldots, 2n-1\} \setminus A \right \rvert$$ denote the number of odd integers missing from $A$, so that any $A \subseteq \{1, 2, \ldots, 2n\}$ with at least $n+4$ elements contains at least $t+4$ even integers. We further define the subsets $E_1, E_2, E_3 \subseteq A$ of even integers contained in the intervals $$[2, 4t + 4], \qquad [4t + 6, 2n - 4t - 4] \qquad \text{and} \qquad [2n - 4t - 2, 2n]$$ respectively. We then split up the argument in five distinct and increasingly specific cases.

\[
\begin{array}{ll}
\text{Case I:}
    & E_2 \neq \varnothing, \\ [2mm]
\text{Case II:}
    & n \le 6t+5, \\ [2mm]
\text{Case III:}
    & E_2 = \varnothing, n \ge 6t+6, \text{ and } E_3 = \varnothing, \\ [2mm]
\text{Case IV:}
    & E_2 = \varnothing, n \ge 6t+6, E_3 \neq \varnothing \text{ and } \lvert E_1 \rvert \le 1, \\ [2mm]
\text{Case V:}
    & E_2 = \varnothing, n \ge 6t+6, E_3 \neq \varnothing \text{ and } \lvert E_1 \rvert \ge 2.
\end{array}
\]

As for the first case, similar observations in regards to this middle set $E_2$ can be found in the proofs of~\cite[Theorem 2]{ces} and~\cite[Theorem 8]{wvd}. For the second case we recall~\cite[Lemma 7]{wvd}, which quantifies the fact that $A$ contains the pairwise sums of (not necessarily positive) integers if $A$ contains sufficiently many even integers in a short interval, and here we will give a similar argument that works for positive integers. However, before we deal with any of the five upper bound cases, let us first look at the lower bound.

\section{Lower bound} \label{lower}
In this section we show that $h_4(n) \ge 4$, which is essentially already done in~\cite[Theorem 3]{wvd}. We give the full proof for completeness.

\begin{lemma}
For all $n \ge 3$ we have $h_4(n) \ge 4$. That is, for all $n \ge 3$ there exists a set $A \subseteq \{1, 2, \ldots, 2n\}$ with $n+3$ elements for which no four distinct positive integers $b_1, b_2, b_3, b_4$ exist such that $b_i + b_j \in A$ for all $1 \le i < j \le 4$.
\end{lemma}

\begin{proof}
We use the set
\begin{equation} \label{adef2}
A := \{1, 3, \ldots, 2n-1\} \cup \{2, 2n-2, 2n\}
\end{equation}
with $n+3$ elements, and claim that no four positive integers exist whose pairwise sums are contained in $A$. \\

Let $b_1, b_2, b_3, b_4$ be four distinct positive integers and assume by contradiction that we have $b_i + b_j \in A$ for all $1 \le i < j \le 4$. If three of the $b_i$ have the same parity, then that would give us three pairwise even sums, while no sum can be equal to $2$. Hence, let us assume that $b_1 < b_2$ are even and $b_3 < b_4$ are odd. Then $\{b_1 + b_2, b_3 + b_4\} \subseteq \{2n-2, 2n\}$ which implies that both $b_2$ and $b_4$ are at least $n$. Since they are not equal to one another, the sum $b_2 + b_4$ is larger than $2n$ and in particular not in $A$.
\end{proof}

For the rest of this paper we will focus on the upper bound, and prove Proposition~\ref{mainprop}.

\section{Upper bound}
\subsection{Case I: Evens in the middle}
As mentioned, we start off with the case where there exists an even integer $$a \in E_2 = A \cap [4t + 6, 2n - 4t - 4] \cap 2\mathbb{N}.$$

\begin{proof}[Proof of Case I]
First assume $a \le n+1$. Then for all $1 \le m < \frac{a}{4}$, consider the four distinct positive integers $$b_1 := 1, \qquad b_2 := a-1, \qquad b_3 := 2m, \qquad b_4 := a-2m.$$ Then $b_1 + b_2 = b_3 + b_4 = a$, while the other four pairwise sums are all odd and contained in $\{1, 3, \ldots, 2n-1\}$. Indeed, $b_1$ and $b_2$ are odd, $b_3$ and $b_4$ are even, and $$b_i + b_j \le b_2 + b_4 \le 2a - 3 \le 2n-1.$$ Moreover, for distinct $m$, these other four pairwise sums are disjoint from one another. Since there are $\lfloor \frac{a-1}{4} \rfloor \ge t+1$ possibilities for $m$, while $A$ only misses $t$ odd integers, there must be at least one $m$ for which all pairwise sums are contained in $A$. \\

When $a \ge n+2$, we instead consider the four positive integers $$b_1 := a-n, \qquad b_2 := n, \qquad b_3 := a-n-1+2m, \qquad b_4 := n+1-2m,$$ this time with $1 \le m < \frac{2n+2-a}{4}$. To see that they are distinct, we have $b_1 < b_2$ and, by $m < \frac{2n+2-a}{4}$, $b_3 < b_4$, while $\{b_1, b_2\} \cap \{b_3, b_4\} = \varnothing$ due to parity reasons. We similarly have $b_1 + b_2 = b_3 + b_4 = a$, with the other pairwise sums once again odd and contained in $\{1, 3, \ldots, 2n-1\}$. By the same argument as before, there exists an $m$ for which all pairwise sums are in $A$, as the number of possibilities is $\left \lfloor \frac{2n+1-a}{4} \right \rfloor \ge t+1$.
\end{proof}

\subsection{Case II: Many missing odds} \label{many}
Recall that we assume $n \le 6t+5$ in this case, which is the only case in which we need to assume that $n$ is sufficiently large. We first prove a version of~\cite[Lemma 7]{wvd} that works for positive integers, using $$f_3(x) := \sqrt{2x + \frac{9}{4}} + \frac{3}{2} \qquad \text{and} \qquad f_{4}(x) := \sqrt{xf_3(x) + \frac{1}{4}} + \frac{1}{2}.$$ With these functions it is quickly verified algebraically that, for all $x \ge 0$, we have the (in)equalities
\begin{align*}
f_3(x) &\ge 3, \\
\frac{f_3(x)(f_3(x)-3)}{2} &= x, \\
f_4(x) &\ge 1, \\
f_4(x)(f_4(x)-1) &= xf_3(x).
\end{align*}
In particular, as the relevant functions are increasing, we obtain the following inequalities.

\begin{lemma} \label{abcformula}
Let $x, y, z$ be non-negative real numbers. Then $$y > f_3(x) \qquad \text{implies} \qquad \frac{y(y-1)}{2} > x + y,$$ while $$z > f_4(x) \qquad \text{implies} \qquad z(z-1) > xf_3(x).$$
\end{lemma}

We need two more intermediate lemmas.

\begin{lemma} \label{splussix}
Let $U$ be a finite set of integers, let $w$ be a non-zero integer, and let $l \in \{0, 1, 2, 3\}$ be the number of elements of $\{w, 2w, 4w\}$ that occur as a difference between two elements of $U$. If each of the equations $$u - v = 3w \qquad \text{and} \qquad u - v = 6w$$ has at most two solutions $(u, v) \in U^2$, then there are at most $\lvert U \rvert + l(l-1)$ pairs $(u, v) \in U^2$ with $u - v \in \{w, 2w, 4w\}$.
\end{lemma}

\begin{proof}
Let us define the three sets
\begin{align*} 
U_1 &:= \{u \in U : u+w \in U\}, \\
U_2 &:= \{u \in U : u-2w \in U\}, \\
U_3 &:= \{u \in U : u+4w \in U\}.
\end{align*}

The number of pairs $(u, v) \in U^2$ with $u - v \in \{w, 2w, 4w\}$ is then equal to the sum $\lvert U_1 \rvert + \lvert U_2 \rvert + \lvert U_3 \rvert$, which is upper bounded by $$\lvert U_1 \cup U_2 \cup U_3 \rvert + \lvert U_1 \cap U_2 \rvert + \lvert U_1 \cap U_3 \rvert + \lvert U_2 \cap U_3 \rvert,$$ thanks to the inclusion-exclusion principle. Now, $\lvert U_1 \cup U_2 \cup U_3 \rvert \le \lvert U \rvert$, so it suffices to calculate the cardinalities of the three pairwise intersections. As every element in either $U_1 \cap U_2$ or $U_1 \cap U_3$ gives rise to a solution of $u - v = 3w$, while every element in $U_2 \cap U_3$ gives rise to a solution of $u - v = 6w$, we find $\lvert U_i \cap U_j \rvert \le 2$ for all $1 \le i < j \le 3$, by the assumption in the lemma statement. Moreover, at most one of these pairwise intersections can be non-empty for $l = 2$, while all three are empty for $l \in \{0, 1\}$. Hence, regardless of whether $l = 0, 1, 2$ or $3$, we deduce 
\begin{equation*}
\lvert U_1 \cap U_2 \rvert + \lvert U_1 \cap U_3 \rvert + \lvert U_2 \cap U_3 \rvert \le l(l-1). \qedhere
\end{equation*}
\end{proof}

By combining Lemma~\ref{abcformula} and~\ref{splussix}, we can prove a version of~\cite[Lemma 7]{wvd} for positive integers.

\begin{lemma} \label{lotsofevens}
Let $E$ be a set of positive even integers. If $$\lvert E \rvert > f_4(\max E - \min E),$$ then there exist four distinct positive integers $b_1, b_2, b_3, b_4$ whose six pairwise sums $b_i + b_j$ are all contained in $E$.
\end{lemma}

\begin{proof}
With $E$ as in the statement of the lemma, let $E_0 \subset E$ be a largest subset of $E$ for which there exists a positive (even) integer $d_1$ such that $e + d_1 \in E$ for all $e \in E_0$. Then we claim that $$\lvert E_0 \rvert > f_3(\max E - \min E) \ge f_3(\max E_0 - \min E_0).$$ The second inequality is easily seen by the fact that $f_3$ is an increasing function, while the first inequality follows from the pigeonhole principle: there are $\frac{\max E - \min E}{2}$ possible (positive) differences between elements of $E$, while there are $\frac{\lvert E \rvert (\lvert E \rvert - 1)}{2}$ pairs $e_1 < e_2 \in E$. Since $$\frac{\lvert E \rvert (\lvert E \rvert - 1)}{2} > f_3(\max E - \min E) \left(\frac{\max E - \min E}{2}\right)$$ by Lemma~\ref{abcformula}, there must be a difference that occurs more than ${f_3(\max E - \min E)}$ times. \\

With $E_0 \subset E$ and $d_1$ as above, we now aim to show that one can find positive integers $d_0, d_2, d_3$ with $d_1, d_2, d_3$ all distinct such that $$\{d_0, d_0+d_2, d_0+d_3, d_0+d_2+d_3 \} \subset E_0.$$ In analogy with~\cite[Lemma A]{ces}, if such integers exist, then by the assumption $E_0 + d_1 \subset E$ it is quickly verified that we can take $$b_1 := \frac{d_0}{2}, \hspace{12pt} b_2 := \frac{d_0}{2} + d_1, \hspace{12pt} b_3 := \frac{d_0}{2} + d_2, \hspace{12pt} b_4 := \frac{d_0}{2} + d_3,$$ to finish the proof. \\

As in the statement of Lemma~\ref{splussix}, let $l \in \{0, 1, 2, 3\}$ be the number of elements of $\left\{\frac{d_1}{2}, d_1, 2d_1 \right\}$ that occur as a difference between two elements of $E_0$. If $\frac{d_1}{2}, d_1$ and $2d_1$ occur more than a combined $\lvert E_0 \rvert + l(l-1)$ number of times as differences between elements in $E_0$, then there exists a $d_2 \in \left \{\frac{3d_1}{2}, 3d_1 \right\}$ that occurs at least three times, by applying Lemma~\ref{splussix} with $U := E_0$ and $w := \frac{d_1}{2}$. Otherwise, as before, we note there are $\frac{\max E_0 - \min E_0}{2}$ possible differences between elements of $E_0$, while there are $\frac{\lvert E_0 \rvert (\lvert E_0 \rvert - 1)}{2}$ pairs $e_1 < e_2 \in E_0$. From the inequalities
\begin{align*}
\frac{\lvert E_0 \rvert (\lvert E_0 \rvert - 1)}{2} &> (\max E_0 - \min E_0) + \lvert E_0 \rvert \\
&\ge 2 \left(\frac{\max E_0 - \min E_0}{2} - l\right) + \lvert E_0 \rvert + l(l-1),
\end{align*}
the first of which is implied by Lemma~\ref{abcformula}, it still follows by the pigeonhole principle that there exists a difference $d_2 \notin \left\{\frac{d_1}{2}, d_1, 2d_1 \right\}$ that occurs at least three times. Hence, let $e_1 < e_2 < e_3 \in E_0$ and $d_2 \notin \left\{\frac{d_1}{2}, d_1, 2d_1 \right\}$ be such that $e_i + d_2 \in E_0$ for $1 \le i \le 3$. Then we claim that at least one of $$e_2 - e_1, \qquad e_3 - e_1, \qquad e_3 - e_2$$ must be different from both $d_1$ and $d_2$. To see this, we note that if two of them are equal to $d_1$, then the third one is equal to either $\frac{d_1}{2}$ or $2d_1$, which are both different from $d_2$ by assumption. Since the argument is the same with $d_1$ and $d_2$ swapped, this proves the claim. \\

To finish the proof, let $d_3 := e_j - e_i \notin \{d_1, d_2\}$ be such a (positive) difference, and set $d_0 := e_i$. Then we conclude $$\{d_0, d_0+d_2, d_0+d_3, d_0+d_2+d_3 \} = \{e_i, e_i + d_2, e_i + (e_j - e_i), e_i + d_2 + (e_j - e_i) \} \subset E_0,$$ as desired.
\end{proof}

We are now finally ready to prove Case II.

\begin{proof}[Proof of Case II]
Assume $n \ge 3000$ and let $2c_1 < 2c_2 < \cdots < 2c_r$ be the even elements of $A \cap [2, 2n - 6t]$. We further write $s$ for the number of even integers of $A$ contained in $[2n - 6t + 2, 2n]$, which gives $r+s \ge t+4$. Now, if there exists an $i$ with $1 \le i \le r-2$ and $c_i + c_{i+1} > c_{i+2}$, we set $$b_1 := c_i + c_{i+1} - c_{i+2}, \hspace{12pt} b_2 := c_i + c_{i+2} - c_{i+1}, \hspace{12pt} b_3 := c_{i+1} + c_{i+2} - c_{i}, \hspace{12pt} b_4 := m,$$ where $m$ is an integer with $1 \le m \le 6t+2$ and $m \not \equiv b_1 \pmod{2}$. All $b_i$ are positive by the assumption $c_i + c_{i+1} > c_{i+2}$, while it is clear that the sums $b_1 + b_2, b_1 + b_3, b_2 + b_3$ are all contained in $A$. Moreover, all other pairwise sums are odd, while they are at most $$b_3 + b_4 \le (c_{r-1} + c_r - c_1) + m \le \big((n - 3t - 1) + (n - 3t) - 1\big) + (6t+2) = 2n.$$ With these definitions, every odd integer missing from $A$ rules out at most three $m$; at most once as a sum involving $b_1$, at most once as a sum involving $b_2$ and at most once as a sum involving $b_3$. Since $A$ misses $t$ odd integers, and there are $3t+1$ options for $m$, we are done if such an $i$ with $c_i + c_{i+1} > c_{i+2}$ exists. On the other hand, if no such $i$ exists, it follows by induction that $c_i \ge F_{i+1}$ for all $i$, where $F_i$ is the $i$th Fibonacci number. In particular, since $F_{i+1} \ge \frac{\varphi^i}{\sqrt{5}}$, we see that $$3t+5 \ge n - 3t \ge c_r \ge \frac{\varphi^r}{\sqrt{5}}$$ or $$r \le \frac{\log(3t+5) + \log \sqrt{5}}{\log \varphi},$$ implying $$s \ge t + 4 - \frac{\log(3t+5) + \log \sqrt{5}}{\log \varphi}.$$ We can then apply Lemma~\ref{lotsofevens} to finish off Case II, as long as 
\begin{equation} \label{ffour}
t + 4 - \frac{\log(3t+5) + \log \sqrt{5}}{\log \varphi}> f_4(6t-2).
\end{equation}
One can check that this inequality does indeed hold for all $t \ge 500$, which is satisfied thanks to the assumption $6t+5 \ge n \ge 3000$.
\end{proof}

\subsection{All evens in the lower interval} \label{allinthelower}
By assumption, all evens in $A$ are contained in $E_1 \subset [2, 4t+4]$, and one can check that applying Lemma~\ref{lotsofevens} would work for $t \ge 128$ in this case. The problem is that the assumption $n \ge 6t+6$ does not let us reduce this to a finite computation. Fortunately, for the interval $E_1$ there is an alternative proof method that we have already seen before. 

\begin{proof}[Proof of Case III]
Let $$2 \le 2c_1 < 2c_2 < \cdots < 2c_{t+4} \le 4t+4$$ be $t+4$ even elements of $A$. Then there must be an $i$ with $1 \le i \le t+2$ such that $c_i + c_{i+1} > c_{i+2}$, as otherwise $c_{i}$ would, for all $i$, be at least as large as $F_{i+1}$ contradicting $c_{t+4} \le 2t+2 < F_{t+5}$. Here, the final inequality quickly follows by induction. And, just as in Case II, with $c_i + c_{i+1} > c_{i+2}$ we set $$b_1 := c_i + c_{i+1} - c_{i+2}, \hspace{12pt} b_2 := c_i + c_{i+2} - c_{i+1}, \hspace{12pt} b_3 := c_{i+1} + c_{i+2} - c_{i}, \hspace{12pt} b_4 := m,$$ for some $m \not \equiv b_1 \pmod{2}$ with $1 \le m \le 6t+2$. The proof is then finished analogously.
\end{proof}

\subsection{Almost all evens in the upper interval} \label{almostallintheupper}
While Case III had the cardinality of $E_1$ large, Case IV applies the fact that $\lvert E_3 \rvert$ is large. Once again, note that we cannot directly apply Lemma~\ref{lotsofevens} as we do not have a lower bound on the value of $t$. However, once again we are able to bypass that issue.

\begin{proof}[Proof of Case IV]
If $\lvert E_1 \rvert \le 1$ while $E_2 = \varnothing$, then $\lvert E_3 \rvert$ must be at least $t+3$. In particular, there are at least three distinct even integers $$2c_1, 2c_2, 2c_3 \in A \cap [2n - 4t - 2, 2n - 2t].$$ For a third time, we then choose $$b_1 := c_1 + c_2 - c_3, \hspace{12pt} b_2 := c_1 + c_3 - c_2, \hspace{12pt} b_3 := c_2 + c_3 - c_1, \hspace{12pt} b_4 := m,$$ for some $m$ with $1 \le m \le 6t+2$ and $m \not \equiv b_1 \pmod{2}$. In this case we do have to verify that $b_1, b_2, b_3$ are all positive, but this follows from the fact that they are at least $$(n - 2t - 1) + (n - 2t) - (n - t) = n - 3t - 1 > 0$$ by the assumption $n \ge 6t+6$. Furthermore, while the pairwise sums that do not involve $b_4$ are contained in $A$, for $i \in \{1, 2, 3\}$ we also have $$b_i + b_4 \le (n - t) + (n - t - 1) - (n - 2t - 1) + (6t + 2) = n + 6t + 2 \le 2n,$$ so that all pairwise sums are between $1$ and $2n$. As before, every odd integer missing from $A$ rules out at most three values of $m$; at most once as a sum involving $b_1, b_2, b_3$ respectively. Since $b_i + b_4$ is odd in every case and there are $3t+1$ possibilities for $m$, this finishes the proof.
\end{proof}

\subsection{Evens in both outer intervals} \label{someineither}
We now consider the fifth and final case.

\begin{proof}[Proof of Case V]
Let $a \in E_3$ and, since $\lvert E_1 \rvert \ge 2$, let $2l$ be an even integer in $E_1$ with $l \ge 2$. Let $\eta \in \{1, 2\}$ be such that $l + \eta$ is odd and consider $$b_1 := l - \eta, \qquad b_2 := l + \eta, \qquad b_3 := a - 2 \left \lfloor \frac{a}{4} \right \rfloor - 2m, \qquad b_4 := 2 \left \lfloor \frac{a}{4} \right \rfloor + 2m,$$ with $1 \le m \le 2t+1$. Clearly, $b_1 + b_2$ and $b_3 + b_4$ both belong to $A$, while we claim that all the other pairwise sums are contained in $\{1, 2, \ldots, 2n\}$ and the $b_i$ are all positive. The second claim is easily checked for $b_i \neq b_3$, while $$\frac{a}{2} \ge n - 2t - 1 \ge 4t+5,$$ so that $$b_3 = a - 2 \left \lfloor \frac{a}{4} \right \rfloor - 2m \ge \frac{a}{2} - 4t-2 \ge 3.$$ As for the first claim, the largest of these pairwise sums is $$b_2 + b_4 \le (2t+4) + (n + 4t + 2) = n + 6t + 6 \le 2n.$$ Now, every odd integer missing from $A$ rules out at most two $m$; at most once as a sum involving $b_1$ and at most once as a sum involving $b_2$. Since $A$ misses $t$ odd integers, and there are $2t+1$ options for $m$, this finishes the proof.
\end{proof}

\section{Final remarks} \label{finalthoughts}
In the proof of Case II one can instead first assume $n \le 5t+4$, in which case $n \ge 2500$ would already suffice to give $t \ge 500$. For the remaining $n$ with $5t+5 \le n \le 6t+6$ essentially the same proof still works, but changing the intervals $A \cap [2, 2n - 6t]$ and $A \cap [2n - 6t + 2, 2n]$ to $E_1$ and $E_3$ respectively lowers the right-hand side of Equation~\eqref{ffour} to $f_4(4t+2)$. This prevents having to check $2500 \le n \le 3000$ computationally. However, as we mentioned before, this and further improvements are not required to reach all the way down to $n = 6$,\footnote{Note that $h_4(5) = 5$, due to the example $A := \{1, 2, 3, 4, 6, 7, 8, 9, 10 \}$.} as the constant in Proposition~\ref{mainprop} is already small enough to computationally verify the remaining $n$. In particular, the Lean file that can be found at~\cite{git1} also includes all necessary computations to unconditionally prove Theorem~\ref{main}.

\end{document}